\documentclass[11pt, reqno,amscd, psamsfonts]{amsart}
\usepackage{amsthm, amsmath}
\usepackage{amssymb}
\usepackage{amscd}
\usepackage{graphicx}
\usepackage{amsfonts}
\newtheorem{theorem}{Theorem}[section]

\newtheorem{definition}[theorem]{Definition}

\newtheorem{hypothesis}[theorem]{Hypothesis}
\newtheorem{lemma}[theorem]{Lemma}

\newtheorem{proposition}[theorem]{Proposition}

\def\arrowdown#1#2{\Big\downarrow \rlap{$\vcenter{\hbox{$\scriptstyle#2$}}$}
{\hbox to -10pt{\hss{$\vcenter{\hbox{$\scriptstyle#1$}}$}}}}
\def\arrowup#1#2{\Big\uparrow \rlap{$\vcenter{\hbox{$\scriptstyle#2$}}$}
{\hbox to -10pt{\hss{$\vcenter{\hbox{$\scriptstyle#1$}}$}}}}

\begin{document}
\title [Morphisms to $\mathbb{P}^1$]{Finite morphisms from curves over Dedekind rings to $\mathbb{P}^1$.}

\author[T. Chinburg]{T. Chinburg*}\thanks{*Supported by NSF Grant \# DMS08-01030}
\address{Ted Chinburg, Dept. of Math\\Univ. of Penn.\\Phila. PA. 19104, U.S.A.}
\email{ted@math.upenn.edu}

\author[G. Pappas]{G. Pappas\dag}
\thanks{\dag Supported by NSF Grant \# DMS08-02686}
\address{Georgios Pappas, Dept. of Math\\Michigan State Univ.\\East Lansing, MI 48824, U.S.A.  }
\email{pappas@math.msu.edu}

\author[M. Taylor]{M. J. Taylor\ddag}
\thanks{\ddag Supported by a Royal Society Wolfson Merit Award}
\address{Martin J. Taylor, School of Mathematics\\ University of Manchester\\ Manchester, M60 1QD, U.K.}
\email{Martin.Taylor@manchester.ac.uk}


\date{\today}

\begin{abstract}
A theorem of B. Green states that if $A$ is a Dedekind ring whose fraction field is
a local or global field, every normal projective curve over $\mathrm{Spec}(A)$ has a finite morphism
to $\mathbb{P}^1_A$.  We give  a different proof of a variant of this result using  intersection theory
and work of Moret-Bailly.
\end{abstract}

\maketitle

\maketitle

\section{Introduction}
\label{s:intro}
\setcounter{equation}{0}

In this paper we will use intersection theory to prove a variant of a theorem first proved by B. Green.
We will  make the following hypotheses:
\begin{hypothesis}
\label{hyp:start} Let $A$ be a excellent Dedekind ring such that:
\begin{enumerate}
\item[i.]  The residue field of each maximal ideal of $A$ is an algebraic extension of a finite field.
\item[ii.]  If $A'$ is the normalization of $A$ in a finite extension $K'$ of the fraction
field $K$ of $A$, then $\mathrm{Pic}(A')$ is a torsion group.
\end{enumerate}
\end{hypothesis}
\noindent

   \begin{theorem}
\label{thm:niceflat}
Suppose that $\mathcal{Y}$ is a normal scheme over $\mathrm{Spec}(A)$ whose structure morphism
 is flat and projective with fibers of dimension $1$. Then there is a finite  flat morphism $\pi:\mathcal{Y} \to \mathbb{P}^1_{A}$ over $\mathrm{Spec}(A)$.
\end{theorem}

This theorem was proved by Green, see \cite[Theorem 2]{Green},  when $K$ is a local or global field.  (See also \cite{Green2}.)
The proof of Theorem 2 in \cite{Green} is written in the language of valuation theory, and follows from a more general result giving sufficient conditions for
a family of valuations on the function field of $\mathcal{Y}$ to be principal.

Since Theorem 1.2 is a geometric result, it is natural to seek an entirely geometric proof. In this paper shall provide such a proof using intersection theory and the work of
Moret-Bailly in \cite{MB}. To give a little more insight into the structure of the proof, we remark that first step is to  show in \S \ref{s:horiz} that the result follows
from the existence of effective horizontal  linearly equivalent ample divisors
$\mathcal{D}_1$
and $\mathcal{D}_2$ on $\mathcal{Y}$
which do not intersect.  Note that  if there is a finite morphism
$\mathcal{Y} \to \mathbb{P}^1_A$, then pulling back the divisors
on $\mathbb{P}^1_A$ associated to homogeneous coordinates $x_0$
and $x_1$ results in such $\mathcal{D}_1$ and $\mathcal{D}_2$.

We then use results of
Moret-Bailly
\cite{MB} to produce an element
$f$ of the function field $K(\mathcal{Y})$ for which $\mathrm{div}_{\mathcal{Y}}(f) = \mathcal{D}_1 - \mathcal{D}_2$ for some
$\mathcal{D}_i$ of the above kind.   Moret-Bailly's method does not lead directly to $f$ for which
$\mathrm{div}_{\mathcal{Y}}(f)$ has no vertical components.  Instead  we produce
a finite set $\{f_i\}_i$ of functions for which the horizontal parts of $\mathrm{div}_{\mathcal{Y}}(f_i)$
are of the desired kind, and for which the vertical parts of the $\mathrm{div}_{\mathcal{Y}}(f_i)$
have the following property.  The vertical parts as $i$ varies generate a subgroup of finite index in the subgroup of the divisor class group of $\mathcal{Y}$ generated by
divisors contained in the reducible fibers of $\mathcal{Y}$ over $\mathrm{Spec}(A)$.
We then use the negative semi-definiteness of the intersection
pairing in fibers to show that a constant times a product of positive integral
powers of the $f_i$ has  a divisor $\mathcal{D}_1 - \mathcal{D}_2$  of the required kind.

In conclusion, we note that Theorem 1.2 and B. Green's results in \cite{Green} or \cite{Green2} are in fact slightly different. Each covers cases that the other does not. His results apply to rings $A$ not satisfying Hypothesis \ref{hyp:start}. Whereas our theorem applies, for example, to the ring $A = \mathbb{Z}[\mu_{p^\infty}][1/p]$ obtained by adjoining to $\mathbb{Z}$ all $p$-power
roots of unity in an algebraic closure of $\mathbb{Q}$ and by then inverting the prime number $p$.

\medbreak
\noindent {\bf Acknowledgements:}
The authors would like to thank
M. Matignon for pointing out the prior work of B. Green after a preliminary version of this
manuscript had been circulated.  The authors would also like to thank J.-B. Bost for
useful discussions.
\medbreak

\section{Horizontal divisors}
\label{s:horiz}
\setcounter{equation}{0}

\begin{lemma}
\label{lem:allright}
To prove Theorem \ref{thm:niceflat}, it will suffice to show that when $\mathcal{Y}$ is connected, there
are ample, effective linearly equivalent horizontal divisors $\mathcal{D}_1$ and $\mathcal{D}_2$ on $\mathcal{Y}$
such that $\mathcal{D}_1 \cap \mathcal{D}_2 = \emptyset$.
\end{lemma}

\begin{proof}  Since $Y$ is normal it is the disjoint union of its connected components, so
we can reduce to the case in which $\mathcal{Y}$ is connected.
Given $\mathcal{D}_1$ and $\mathcal{D}_2$ as in the Lemma, we can replace
each of these divisors by a high integral multiple of themselves to be able
to assume that there is a projective embedding $\mathcal{Y} \to \mathbb{P}^n_A$
and hyperplanes $H_1$ and $H_2$ in $\mathbb{P}^n_A$ such that $H_i \cap \mathcal{Y} = \mathcal{D}_i$ for
$i = 1, 2$.  The fact that $\mathcal{D}_1 \cap \mathcal{D}_2 = \emptyset$ implies that $\mathcal{Y}$
is contained in the open set $U = \mathbb{P}^n_A - (H_1 \cap H_2)$.  Let
$h_i(x)$ be a linear form in the homogenous coordinates $x = (x_0;\ldots;x_n)$ of $\mathbb{P}^n_A$
such that $H_i$ is the zero locus of $h_i(x)$.  There is a morphism $f:U \to \mathbb{P}^1_A$ defined
by $x = (x_0;\ldots;x_n) \to (h_1(x);h_2(x))$.

We first show that
the restriction $f_\mathcal{Y}:\mathcal{Y} \to \mathbb{P}^1_A$ of $f$ to $\mathcal{Y}$ is quasi-finite.
Let $c$ be a point of $\mathrm{Spec}(A)$, and let $Z$ be the reduction
of an irreducible component of the fiber $\mathcal{Y}_c$ of $\mathcal{Y}$ over $c$.  Since $\mathcal{Y}_c$
has finitely many irreducible components, it will suffice
to show that the restriction $f_Z:Z \to \mathbb{P}^1_A$ of $f$  is quasi-finite.  If
$H_1$ contains $Z$, then $Z - (Z \cap U) = Z \cap H_1 \cap H_2 = Z \cap H_2 \ne \emptyset$
since $Z$ is projective and $H_2$ is a hyperplane in $\mathbb{P}^n_A$.  This
is a contradiction, and similarly $H_2$ cannot contain $Z$.  Thus $h(x) = h_2(x)/h_1(x)$
defines a non-zero rational function on $Z$, and it will suffice to show that $h(x)$ is not
in the field of constants $\ell(Z)$ of the function field $k(Z)$ of $Z$.

Suppose first that $c$ is the generic point of $\mathrm{Spec}(A)$.  Since $\mathcal{Y}$
is normal and connected,  its generic fiber $Y$ is regular and irreducible.
Thus $Z = Y$.  If $h(x) \in \ell(Z)$, the ample divisors which $\mathcal{D}_1$ and $\mathcal{D}_2$
determine on $Y$ are equal, contradicting  $\mathcal{D}_1 \cap  \mathcal{D}_2 = \emptyset$.

Suppose now that $c$ is a closed point of $\mathrm{Spec}(A)$.  Then $\ell(Z)$ is a finite extension of the residue field
$k(c)$ of $c$. By Hypothesis \ref{hyp:start}(i), $k(c)$ is an algebraic extension of a finite field, so $\ell(Z)$ is also such an extension.  Therefore if $h(x) \in \ell(Z)$, the fact that $h(x)$ is non-zero implies that there is an integer $m > 0$
such that $h(x)^m = (h_2(x)/h_1(x))^m = 1$ in $\ell(Z)$.   Hence the zero locus of $h_2(x)^m - h_1(x)^m$
contains a dense open subset of $Z$, and thus all of $Z$.
Since $Z$ is projective over $A$, there is a point $z$ in $Z \cap H_2$.
Then $h_2(z) = 0$ and $h_2(z)^m - h_1(z)^m = 0$, so $h_1(z)^m = 0$
and $h_1(z) = 0$.  This implies $z \in Z \cap H_0 \cap H_1$, which contradicts
$\mathcal{D}_1 \cap  \mathcal{D}_2 = \emptyset$.
Thus $f_\mathcal{Y}$ is quasi-finite.

Since the projective morphism $\mathcal{Y} \to \mathrm{Spec}(A)$
factors through $f_\mathcal{Y}$, we see that $f_\mathcal{Y}$ must be projective.  By \cite[Ex. III.11.2]{H}
a quasi-finite projective morphism is finite. By \cite[Thm. 38, p. 124]{Mat}, $\mathcal{Y}$ is Cohen-Macaulay
because it noetherian and normal of dimension two.  By \cite[Thm. 46, p. 140]{Mat}, a finite
morphism from a Cohen-Macaulay scheme to a regular scheme is flat.  Hence $f_\mathcal{Y}:\mathcal{Y} \to \mathbb{P}^1_A$ is finite and flat, so Lemma \ref{lem:allright}
is proved.
\end{proof}

\section{Intersection numbers and ample divisors}
\label{s:=intample}
\setcounter{equation}{0}

In this section we define some notation and we recall some well known results about
intersection numbers and ample divisors.  We will assume throughout that $\mathcal{Y}$
is connected.

\begin{definition}
\label{def:defins}   Let $S_v = S_v(\mathcal{Y})$ be the set of irreducible components
of the fiber $\mathcal{Y}_v = k(v) \otimes_A \mathcal{Y}$ of $\mathcal{Y}$ over $v \in \mathrm{Spec}(A)$.  Define $\mathcal{Y}_v^{red}$ to be the reduction of $\mathcal{Y}_v$.  Let $Y = K \otimes_A \mathcal{Y}$ be the general fiber of $\mathcal{Y}$.
\end{definition}

\begin{definition}
\label{def:inter}Suppose $E$ is a Cartier divisor on $\mathcal{Y}$
and that $C \in S_v$ for some maximal ideal $v$ of $A$.  Let $C^\#$ be the normalization of $C$, and let $i:C^\# \to \mathcal{Y}$ be the composition
of the natural morphism $C^\# \to C$ with the closed immersion $C \to \mathcal{Y}$.  Define
  $$\langle E,C\rangle_v = \mathrm{deg}_{k(v)} i^*(O_{\mathcal{Y}}(E))$$ where $ i^*(O_{\mathcal{Y}}(E))$ is a line bundle on the regular curve $C^\#$ over the residue field
$k(v)$ of $v$.  This pairing may be extended by bilinearity to all Cartier divisors $E$ and to all Weil divisors $C$ in the
free abelian group $W_v$ generated by $S_v$.
\end{definition}

The value of $\langle E,C\rangle$
clearly depends only on the linear equivalence class of $E$.  We will need the following
result.

\begin{lemma}
\label{lem:okay}{\rm(Moret-Bailly)}
A non-zero integral
  multiple of a Weil divisor on $\mathcal{Y}$ is a Cartier divisor.
One may thus extend $\langle E, C \rangle_v$ to all Weil divisors $E$ and all $C \in W_v$
  by linearity in both arguments.  Define $\mathbb{Q}W_v = \mathbb{Q}\otimes_\mathbb{Z} W_v$
  and let $\mathbb{Q}\mathcal{Y}_v$ be the subspace spanned by
  the Weil divisor $\mathcal{Y}_v$.  Then $\langle\ , \ \rangle_v$ gives rise to a negative
  definite pairing
  \begin{equation}
  \label{eq:negdef}
  \langle\ , \ \rangle_v:\frac{\mathbb{Q}W_v}{\mathbb{Q}\mathcal{Y}_v} \times  \frac{\mathbb{Q}W_v}{\mathbb{Q}\mathcal{Y}_v} \to \mathbb{Q}.
  \end{equation}
  Let $\mathcal{T}$ be a horizontal Cartier divisor on $\mathcal{Y}$, and let $T$ be the general
  fiber of $\mathcal{T}$. Then
  \begin{equation}
  \label{eq:fiberint}
  \langle \mathcal{T},\mathcal{Y}_v\rangle_v = \mathrm{deg}_K(T)
  \end{equation}
  for all maximal ideals $v \in \mathrm{Spec}(A)$.
  \end{lemma}

  \begin{proof}  The first assertion is shown in \cite[Lemme 3.3]{MB}. Since $\langle E, C \rangle_v$
  is bilinear over Cartier divisors $E$, it follows that we can extend this pairing to all
  Weil divisors $E$.   The proof of the second assertion concerning (\ref{eq:negdef}) is indicated immediately
  after \cite[eq. (3.5.4)]{MB}. For further details, see  \cite[exp. 1, Prop. 2.6]{Sz} and \cite[\S 2.4, Appendices A.1 and A.2]{Fulton}.  The last assertion is from
  \cite[\S 3.5]{MB}.
  \end{proof}

Note that if $\mathcal{Y}$ is not regular, the extension of $\langle \ , \ \rangle_v$ described in \ref{lem:okay}
  is different in general from the proper intersection pairing considered by Artin in
  \cite[\S 2]{Artin}.

\begin{lemma}
\label{lem:ample}
Suppose $\mathcal{D}$ is an ample Cartier divisor on $\mathcal{Y}$ and that $\mathcal{E}$
is an effective horizontal Cartier divisor.  Then $\mathcal{D} + \mathcal{E}$ is ample.
\end{lemma}

\begin{proof}  By \cite[Prop. III.5.3]{H},
$\mathcal{D} + \mathcal{E}$ is ample if and only if  for each coherent sheaf $\mathcal{F}$ on $\mathcal{Y}$,
there is an integer $n_0(\mathcal{F}) > 0 $ such that
$${\rm H}^i(\mathcal{Y},\mathcal{F}\otimes O_{\mathcal{Y}}(\mathcal{D} + \mathcal{E})^{\otimes n}) = 0$$
for all $n \ge n_0$ and all $i > 0$.  Consider the long exact cohomology sequence associated to
the exact sequence of sheaves
$$0 \to \mathcal{F} \otimes O_{\mathcal{Y}}(\mathcal{D})^{\otimes n} \to \mathcal{F} \otimes O_{\mathcal{Y}}(\mathcal{D} + \mathcal{E})^{\otimes n} \to
 \left (\mathcal{F} \otimes O_{\mathcal{Y}}(\mathcal{D} + \mathcal{E})^{\otimes n}\right )|_{n\mathcal{E}} \to 0.$$
 Because $n\mathcal{E}$ is horizontal, it is affine, and the higher cohomology of
 coherent sheaves on $n \mathcal{E}$ is trivial.  It follows that $\mathcal{D} + \mathcal{E}$ is ample
  because $\mathcal{D}$ is ample.
  \end{proof}
  Note that if $\mathcal{E}$ is allowed to have vertical components, then $\mathcal{D} + \mathcal{E}$
  might have negative degree on some irreducible vertical component of $\mathcal{Y}$.
  Thus the conclusion of Lemma \ref{lem:ample} need not hold for arbitrary effective
  Cartier divisors $\mathcal{E}$.

We will leave the proof of the following Lemma to the reader.

\begin{lemma}
\label{lem:projsimple}
Suppose $T$ is a finite set of closed points of $\mathbb{P}^m_A$ for some integer $m \ge 1$.
Then there is an integer $n \ge 1$ and a homogenous polynomial $f = f(x_0,\ldots,x_m)$
of degree $n$ in homogenous coordinates $(x_0;\ldots;x_m)$ for $\mathbb{P}^m_A$
such that $f$ does not vanish at any point of $T$.
\end{lemma}

\begin{lemma}
\label{lem:nicehoriz}
There is an effective horizontal divisor $\mathcal{D}$ on $\mathcal{Y}$
which is  very ample relative to the structure morphism $\mathcal{Y} \to \mathrm{Spec}(A)$.  Each such
$\mathcal{D}$ intersects every irreducible component of a fiber of $\mathcal{Y}$
over $\mathrm{Spec}(A)$.
\end{lemma}

\begin{proof} Since we have assumed $\mathcal{Y}$ is projective,
there is an effective very ample Cartier divisor $\mathcal{D}$ on $\mathcal{Y}$.
Let $T$ be
a finite set of closed points of $\mathcal{Y}$ which contains a point
on every irreducible component of every reducible fiber of $\mathcal{Y}$
over $\mathrm{Spec}(A)$.  Lemma \ref{lem:projsimple} implies that there
is an effective very ample Cartier divisor which is linearly equivalent to $n\mathcal{D}$
for some $n > 0$
and which contains no point of $T$;  we replace $\mathcal{D}$ by this
divisor.  Now $\mathcal{D}$ can contain no irreducible component
of a reducible vertical fiber of $\mathcal{Y}$ over $\mathrm{Spec}(A)$.
Thus the vertical part of $\mathcal{D}$ is an integral combination
of fibers of $\mathcal{Y}$.
Since $\mathrm{Pic}(A)$ is finite by assumption, we can
now replace $\mathcal{D}$ by $d\mathcal{D} + \mathrm{div}_{\mathcal{Y}}(g)$
for some $g$ in the fraction field $K$ of $A$ and some $d > 0$ to be able to assume
that $\mathcal{D}$ is horizontal, effective and very ample.   Since $\mathcal{D}$
is effective and ample when restricted to every irreducible component of a fiber
of $\mathcal{Y}_v$, it must intersect each such component.
\end{proof}

\section{An application of work of Moret-Bailly}
\label{s:=MBsec}
\setcounter{equation}{0}

 \begin{proposition}
 \label{prop:divprop}
Suppose $\mathcal{Y}$ is connected.  Let $\mathcal{D}$ be a divisor with the properties stated in
Lemma \ref{lem:nicehoriz}.   Let $M$ be the finite set of maximal ideals $v \in \mathrm{Spec}(A)$
such that $\mathcal{Y}_v$ has more than one irreducible component.
 For each $v \in M$, choose an element $C(v)$ of $S_v$. Suppose
  $$\mathcal{Y}_v = \sum_{C \in S_v} n_C C$$
as Weil divisors  for some integers $n_C > 0$.  There is a non-constant function $f$ in the function field $K(\mathcal{Y})$ having the following properties.
 The divisor of $f$ on $\mathcal{Y}$ has the form
 \begin{equation}
 \label{eq:yupf}
 \mathrm{div}_{\mathcal{Y}}(f) = \mathcal{D}_1(f)  - n\mathcal{D} + \sum_{v \in M} E_v \end{equation}
where $0 < n \in \mathbb{Z}$, $E_v$ is a Cartier divisor supported on $\mathcal{Y}_v$ for $v \in M$ and the following is true:
 \begin{enumerate}
 \item[i.]  $\mathcal{D}_1(f)$ is an effective, horizontal Cartier divisor and is equal to the
 Zariski closure of its general fiber $D_1(f)$.   The intersection $\mathcal{D}_1(f) \cap \mathcal{D}$
 is empty.
  \item[ii.]  Suppose $v \in M$ and $C \in S_v$. Then
 \begin{equation}
 \label{eq:urp}  \langle n \mathcal{D},C\rangle_v - \langle \mathcal{D}_1(f),C \rangle_v =  \langle E_v, C\rangle_v \in \mathbb{Z}.
 \end{equation}
 \item[iii.]  Let $m$ be the degree of
 $f$ on the general fiber $Y$.    If $v \in M$ and $C \in S_v$ then
  \begin{equation}
 \label{eq:bounds}
 0 < \langle n\mathcal{D},n_C C \rangle_v < m.
 \end{equation}
\item[iv.] If $v \in M$, the unique component of the special fiber $\mathcal{Y}_v$ which $\mathcal{D}_1(f)$ intersects  is the $C(v)$ we have chosen.
 For $C \in S_v$ one has
\begin{equation}
\label{eq:crank1}
 \langle \mathcal{D}_1(f),C \rangle_v = 0 \quad \mathrm{if}\quad C \ne C(v)\quad \mathrm{and}
 \quad \langle \mathcal{D}_1(f),n_{C(v)} C(v)  \rangle_v = m.
 \end{equation}
\item[v.] For $v \in M$ one has
 \begin{equation}
 \label{eq:crank}
 \langle E_v, n_C C \rangle_v > 0 \quad \mathrm{if} \quad C(v) \ne C \in S_v \quad  \mathrm{and}\quad
  \langle E_v, n_{C(v)} C(v)\rangle_v
 < 0.
 \end{equation}
 \end{enumerate}
 \end{proposition}

 \begin{proof}
We use the construction
 given
 in the proof of \cite[Prop. 3.8]{MB}.  To match the notation used in \cite{MB},
 let $\overline{X} = \mathcal{Y}$.  Define $Z$ to be the union of $\mathcal{D}$
 with
 \begin{equation}
 \label{eq:Vdef}
 V = \cup\{\ C\  :  \ v \in M \quad \mathrm{and}\quad C(v) \ne C \in S_v\} .
 \end{equation}
 In the proof of \cite[Prop. 3.8]{MB}, Moret-Bailly
shows there is an integer $n > 0$ and rational linear combination $\Delta$
 of vertical divisors with the following properties:
 \begin{enumerate}
 \item[i.]  The rational divisor $n(\mathcal{D} + \Delta)$
 is a Cartier divisor.
 \item[ii.] Let $\mathcal{L} = O_{\mathcal{Y}}(n(\mathcal{D} + \Delta))$.
 There is a non-zero global section \hbox{$t \in H^0(\mathcal{Y},\mathcal{L})$} such that
 $t$ generates the stalk of $\mathcal{L}$ at all points of $Z$.
 \end{enumerate}

 On viewing $\mathcal{L}$ as a subsheaf of the function field $K(\mathcal{Y})$
 we may  identify $t$ with
 a function $f \in K(\mathcal{Y})$.  Then
 \begin{equation}
 \label{eq:word}
 \mathrm{ord}_w(f) = -\mathrm{ord}_w(n(\mathcal{D} + \Delta))
 \end{equation}
 at all codimension $1$ points $w$ of $\mathcal{Y}$ lying in $Z = \mathcal{D} + V$.
 Since $V$ and $n\Delta$ are fibral, and $Z$ contains $\mathcal{D}$, we conclude that
 $$\mathrm{div}_{\mathcal{Y}}(f) = \mathcal{D}_1(f)  - n\mathcal{D} + \mathcal{T}$$
 where $\mathcal{T}$ is a fibral divisor and $\mathcal{D}_1(f)$ is an effective, horizontal  divisor having no
 irreducible components in common with $\mathcal{D}$.  It follows that $\mathcal{D}_1(f)$ is the
 Zariski closure of its general fiber $D_1(f)$.

 Write $\mathcal{T}$ as a finite integral combination of irreducible components of
 the $\mathcal{Y}_v$ as $v$ ranges over $\mathrm{Spec}(A)$.  If $v \not \in M$, then $\mathcal{Y}_v^{red}$ is irreducible,
 and a non-zero integral multiple of $\mathcal{Y}_v^{red}$ is the divisor of a non-zero element of $K$ since
  $\mathrm{Pic}(A)$ is torsion by Hypothesis \ref{hyp:start}.  By Lemma \ref{lem:okay}, a non-zero integral
  multiple of a Weil divisor on $\mathcal{Y}$ is a Cartier divisor.  Therefore on replacing  $f$ by $f^m \alpha$ for some $0 \ne \alpha \in K$ and some sufficiently divisible integer $m > 0$, we will have an equality
  of the form in (\ref{eq:yupf}), with $n\mathcal{D}$ and $n\Delta$ being Cartier divisors.

  Since  $t$ generates the stalk of $\mathcal{L}
 = O_{\mathcal{Y}}(n(\mathcal{D} + \Delta))$ at each point of $Z$, $f^{-1}$ is a local
 equation for the Cartier divisor $n(\mathcal{D} + \Delta)$ at each point of $Z$.  From
 (\ref{eq:yupf}), $f^{-1}$ is also a local equation for
 $-\mathcal{D}_1(f)  +  n\mathcal{D} + \sum_{v \in M} E_v$ at all points of $\mathcal{Y}$.  Recall that
 the $E_v$ and $n \Delta$ are vertical Cartier divisors, and  $\mathcal{D}_1(f)$ is a horizontal Cartier divisor with no irreducible components in common with
 $\mathcal{D} \subset Z$. Therefore if $c \in \mathcal{D}_1(f) \cap \mathcal{D} \subset Z$, a local equation  for $\mathcal{D}_1(f)$ in a neighborhood of $c$ would have to be a local equation for $\sum_{v \in M} E_v -n\Delta$. These Cartier divisors
 would then agree in an open neighborhood of $c$ in $\mathcal{Y}$, which  is impossible since one is vertical and the other is horizontal.
Therefore $\mathcal{D}_1(f) \cap \mathcal{D} = \emptyset$, which
completes the proof of (i).

 Statement (ii) is clear from (\ref{eq:yupf}) and the fact that since $\mathrm{div}_{\mathcal{Y}}(f)$ is principal, $\langle \mathrm{div}_{\mathcal{Y}}(f),C\rangle_v = 0$.

 Concerning (iii), we know by (\ref{eq:fiberint}) that
 \begin{equation}
 \label{eq:sumequation}
  \sum_{C \in S_v} \langle n\mathcal{D},n_C C\rangle_v = \langle n\mathcal{D}, \mathcal{Y}_v\rangle_v =  \mathrm{deg}(nD) = m
 \end{equation}
 since $nD$ is the polar part of the divisor $\mathrm{div}_Y(f)$ on the general fiber $Y$ of $\mathcal{Y}$.
We have $\langle  \mathcal{D},C\rangle_v > 0$ for all $C \in S_v$ because $\mathcal{D}$ intersects
each such $C$ by Lemma \ref{lem:nicehoriz}.  Therefore (\ref{eq:bounds}) follows from (\ref{eq:sumequation}) and the fact that $S_v$ has more than one element if
$v \in M$.

To show (iv), suppose $v \in M$ and that $C(v) \ne C \in S_v$.
Since $C \subset Z$ and $C$ is fibral, (\ref{eq:yupf}) and (\ref{eq:word}) imply that the multiplicities of $C$ in $-n(\mathcal{D} + \Delta)$, $-n\Delta$,
$\mathrm{div}_{\mathcal{Y}}(f)$ and $E_v$ must be equal.  We conclude from (\ref{eq:yupf}) and (\ref{eq:word}) that if $\mathcal{D}_1(f)$
intersects a point $c \in C\subset Z$, then $t$ would not be a local generator of the stalk of $\mathcal{L}$
at $c$.  This contradicts condition (ii) in the definition of $t$ following equation (\ref{eq:Vdef}).  Thus $\mathcal{D}_1(f)$ can only intersect the
component $C(v)$ of $\mathcal{Y}_v$.   Now (\ref{eq:fiberint})
 shows
$$\langle \mathcal{D}_1(f),n_{C(v)} C(v)\rangle_v = \langle \mathcal{D}_1(f),\mathcal{Y}_v\rangle_v = \mathrm{deg}(D_1(f)) = m.$$
This shows (\ref{eq:crank1}) and completes the proof of (iv).

 Finally, the inequalities in (\ref{eq:crank}) of part  (v) are a consequence of (\ref{eq:urp}),  (\ref{eq:bounds}) and
 (\ref{eq:crank1}).
 \end{proof}

\section{Controlling vertical divisors}
\label{s:Rumcapdiv}
\setcounter{equation}{0}

\begin{lemma}
\label{lem:fixit}  Let $\mathcal{D}$, $M$ and $S_v$ be as in Proposition \ref{prop:divprop}.  Suppose
$M' \subset M$ and that $C_0(v) \in S_v$ for $v \in M'$.  Then
 there is a function $h \in K(\mathcal{Y})$ such that
 \begin{equation}
 \label{eq:divisornow}
 \mathrm{div}_{\mathcal{Y}}(h) = \mathcal{D}_1 - \mathcal{D}_2 + \sum_{v \in M'} E_v
 \end{equation}
 where $\mathcal{D}_1$ and $\mathcal{D}_2$ are horizontal effective
 divisors which do not intersect, $\mathcal{D}_2$ has the same support as $\mathcal{D}$, $E_v$ is supported on
 $\mathcal{Y}_v$, and for $v \in M'$ and $C' \in S_v$
 we have
 \begin{equation}
 \label{eq:crank2}
 \langle E_v, C' \rangle_v > 0 \quad \mathrm{if} \quad C_0(v) \ne C' \in S_v \quad  \mathrm{and}\quad
  \langle E_v, C_0(v)\rangle_v
 < 0.
 \end{equation}
 \end{lemma}

 \begin{proof} We use induction on the number of elements of $M - M'$.  If $M = M'$,
 the Lemma is shown by Proposition \ref{prop:divprop}.  Suppose now that Lemma \ref{lem:fixit} holds when $M'$ is replaced by $M' \cup \{v_0\}$
   for some $v_0 \in M - M'$.  For each $C \in S_{v_0}$, we thus can find a function
   $h_C$ with the following properties:
 \begin{enumerate}
 \item[i.]  The divisor of $h_C$ is
 \begin{equation}
 \label{eq:divisornownew}
 \mathrm{div}_{\mathcal{Y}}(h_C) = \mathcal{D}_{C,1} - \mathcal{D}_{C,2} + \sum_{v \in M'} E_{C,v} + E_{C,v_0}
 \end{equation}
 where $\mathcal{D}_{C,1}$ and $\mathcal{D}_{C,2}$ are horizontal effective
 divisors which do not intersect, $\mathcal{D}_{C,2}$ has the same support as $\mathcal{D}$, and $E_{C,v}$ is supported on
 $\mathcal{Y}_v$ for $v \in M' \cup \{v_0\}$.
 \item[ii.]
For $v \in M'$ and $C' \in S_v$
 we have
 \begin{equation}
 \label{eq:crank3}
 \langle E_{C,v}, C' \rangle_v > 0 \quad \mathrm{if} \quad C_0(v) \ne C' \quad  \mathrm{and}\quad
  \langle E_{C,v}, C_0(v)\rangle_v
 < 0.
 \end{equation}
 \item[iii.]
 For $C' \in S_{v_0}$ we have
  \begin{equation}
 \label{eq:crank4}
 \langle E_{C,v_0}, C' \rangle_{v_0} > 0 \quad \mathrm{if} \quad C \ne C'  \quad  \mathrm{and}\quad
  \langle E_{C,v_0}, C\rangle_{v_0}
 < 0.
 \end{equation}
 \end{enumerate}

 We claim that there are positive integers $\{a_{C}\}_{C \in S_{v_0}}$ such that
 the divisor $E_{v_0} = \sum_{C \in S_{v_0}} a_{C} E_{C,v_0}$ has the property
 that
 \begin{equation}
 \label{eq:enice}
 \langle E_{v_0},C' \rangle_{v_0} = 0 \quad \mathrm{for \ all} \quad C' \in S_{v_0}.
 \end{equation}
 Before showing this, let us first show how it can be used to complete
 the proof of Lemma \ref{lem:fixit}.

 By Lemma \ref{lem:okay}, the intersection pairing $\langle  \ , \  \rangle_{v_0}$ is negative
 semi-definite on the vector space spanned by $S_{v_0}$.  Hence
 (\ref{eq:enice}) implies that $E_{v_0}$ is a rational multiple of the fiber $\mathcal{Y}_{v_0}$.
 Since $\mathrm{Pic}(A)$ is finite by Hypothesis \ref{hyp:start}(ii), there is a positive
 integer $d$ such that  $d\cdot E_{v_0}$ is the principal (vertical) divisor of a constant $a \in K^*$.  We have
 \begin{equation}
 \label{eq:almostdone}
 \sum_{C \in S_{v_0}} a_{C} \cdot \mathrm{div}_{\mathcal{Y}}(h_C) = \mathcal{D}_1 - \mathcal{D}_2 +
 \sum_{v \in M'} E_{v} + E_{v_0}
 \end{equation}
 where
 \begin{equation}
 \label{eq:dumb}
 \mathcal{D}_i = \sum_{C \in S_{v_0}} a_C \mathcal{D}_{C,i}
 \end{equation}
 and
 \begin{equation}
 \label{eq:youwie}
 E_v = \sum_{C \in S_{v_0}} a_C E_{C,v}
 \end{equation}
 for $v \in M' \cup \{v_0\}$.  The support of each $\mathcal{D}_{C,2}$ equals that of $\mathcal{D}$, and this does not intersect the support of any of the $\mathcal{D}_{C,1}$
 by our induction hypothesis.  It follows that $\mathcal{D}_1$ and $\mathcal{D}_2$
 are effective horizontal divisors which do not intersect, and $\mathcal{D}_2$
 and $\mathcal{D}$ have the same support.   Because of the induction
 hypotheses (\ref{eq:crank3}) and (\ref{eq:crank4}), the fact that all of the
 $a_C$ associated to $C \in S_{v_0}$ are positive integers implies that
condition (\ref{eq:crank3}) holds if we replace $E_{C,v}$ in that condition
by $E_v$.  Now since
$$\mathrm{div}_{\mathcal{Y}}(a) = d \cdot E_{v_0}$$
and $d > 0$ we conclude from (\ref{eq:almostdone}) that the function
\begin{equation}
\label{eq:finalh}
h = a^{-1} \cdot \left ( \prod_{C \in S_{v_0}} h_C^{a_C}\right )^d
\end{equation}
will have all the properties required to show the induction step for Lemma
\ref{lem:fixit} and complete the proof.

It remains to produce positive integers $\{a_{C}\}_{C \in S_{v_0}}$ such
that $$E_{v_0} = \sum_{C \in S_{v_0}} a_{C} E_{C,v_0}$$ has property
 (\ref{eq:enice}), i.e. is perpendicular to every irreducible component $C'$
 of $\mathcal{Y}_{v_0}$.  It will suffice to show that we can do this
 using positive rational numbers $a_C$ since the intersection pairing
 is well defined for all rational linear combinations of fibral divisors.

 Consider the square matrix $W = (W_{C,C'})_{C,C' \in S_v}$ with integral entries
 $$W_{C,C'} = \langle E_{C,v_0}, n(C') C' \rangle$$ where $n(C') > 0$ is the multiplicity
 of $C'$ in the fiber $\mathcal{Y}_v$.  The sum of all the
 entries in the row indexed by $C$ is
 $$\sum_{C' \in S_v} \langle E_{C,v_0}, n(C') C' \rangle_{v_0} = \langle E_{C,v_0}, \sum_{C' \in S_v} n(C') C' \rangle_{v_0} = \langle E_{C,v_0},\mathcal{Y}_{v_0}\rangle_{v_0} = 0$$
 where the last equality is from Lemma \ref{lem:okay}.
 Condition (\ref{eq:crank4}) of the induction
 hypothesis now says that $W$ satisfies the hypotheses of the following
 Lemma, and this Lemma completes the proof of Lemma \ref{lem:fixit}.
 \end{proof}
 \begin{lemma}
 \label{lem:linearalg}
 Suppose $W = (w_{i,j})_{1 \le i,j \le t}$ is a square matrix
 of rational numbers such that the  diagonal (resp. off-diagonal) entries
 are negative (resp. positive) and that the sum of the entries in any row is $0$.
 Then there is a positive rational linear combination of the rows which is the
 zero vector.
 \end{lemma}
 \begin{proof}
 We prove this assertion by ascending induction on the size $t$ of $W$.
  If $t = 1$ then $W$ has to be the zero matrix since the sum of the
 entries in any row of $W$ is trivial.  If $t = 2$ then the rows of $W$
 have the form $(-a,a)$ and $(b,-b)$ for some positive rationals $a$ and
 $b$, so $b$ times the first row plus $a$ times the second is $(0,0)$.
 We now suppose the statement is true for matrices of smaller size
 than $t \ge 3$.

 Without loss of generality, we can multiply the $i$-th row
 of $W$ by $-1/w_{i,i} > 0$ to be able to assume that the diagonal entries
 are all equal to $-1$.  Since every off diagonal entry is positive, every off diagonal
 entry has to be a rational number in the open interval $(0,1)$ because the sum
 of the entries in each row is $0$ and $t \ge 3$.

 Thus when we add $w_{i,t}$ times the last row to the $i^{th}$
 row for $i = 1,\ldots,t-1$, we arrive at a matrix $W' = (w'_{i,j})_{i,j = 1}^t$
 such that $w'_{i, t} = 0$ for $i = 1,\ldots,t-1$.  It is elementary to check that the
the $(t-1)\times (t-1)$ matrix $W'' = (w'_{i,j})_{i,j = 1}^{t-1}$ which
 results from dropping the last row and the last column of $W'$ satisfies
 our induction hypotheses.
 We now conclude by induction that there is a
 positive rational linear combination of the rows of $W''$ which
 equals $0$.  The corresponding linear combination of the rows
 of $W'$ is then also $0$.  Since each of the first $t-1$ rows of $W'$
 is the sum of the corresponding row of $W$ with a positive multiple
 of the last row of $W$, we arrive in this way at the a positive
 linear combination of the rows of $W$ which is the zero vector.
 \end{proof}

\medbreak
\noindent {\bf Completion of the proof of Theorem \ref{thm:niceflat}}
\medbreak
Let $M'$ be the empty set in Lemma \ref{lem:fixit}.  This Lemma now
produces divisors $\mathcal{D}_1$ and $\mathcal{D}_2$ which
are horizontal, effective, disjoint, linearly equivalent, and for which
$\mathcal{D}_2$ has the same support as the horizontal effective
ample divisor $\mathcal{D}$.    Lemma \ref{lem:ample} shows there
is an integer $m > 0$ such that  $m\mathcal{D}_2$ is ample.  So on replacing
$\mathcal{D}_i$ by $m\mathcal{D}_i$ for $i = 1, 2$ we arrive at divisors
of the kind needed in Lemma \ref{lem:allright}, which
finishes the proof.

\end{document}